\documentclass[12pt,american]{article}
\usepackage[T1]{fontenc}
\usepackage{color}
\usepackage{babel}
\usepackage{amsmath}
\usepackage{amsthm}
\usepackage{amssymb}
\usepackage{comment}

\usepackage{datetime}
\usepackage{mathrsfs}
\usepackage{pdfsync}

\usepackage[bookmarks=true,bookmarksnumbered=false,bookmarksopen=false,
 breaklinks=true,pdfborder={0 0 1},backref=false,colorlinks=true]
 {hyperref}
\hypersetup{pdftitle={Identification of the residual term in multiplicative self-decomposition using Fox \texorpdfstring{$H$}{Lg}-functions},
 pdfauthor={Mohamed Erraoui and Jos{\'e}Lu{\'\i}s da Silva},
 pdfsubject={Selfdecomposable distributions and special functions},
 pdfkeywords={Selfdecomposable distributions. \texorpdfstring{$M$}{Lg}-Wright function, exponential distribution, Mellin transform},
pdfborderstyle=,linkcolor=black,citecolor=blue,urlcolor=blue,filecolor=blue}

\makeatletter
\theoremstyle{plain}
\newtheorem{thm}{\protect\theoremname}[section]

\newtheorem{prop}[thm]{\protect\propositionname}

\theoremstyle{remark}
\newtheorem{rem}[thm]{\protect\remarkname}

\theoremstyle{definition}

\providecommand{\theoremname}{Theorem}
\providecommand{\lemmaname}{Lemma}
\providecommand{\corollaryname}{Corollary}
\providecommand{\propositionname}{Proposition}
\providecommand{\conjecturename}{Conjecture}
\providecommand{\criterionname}{Criterion}
\providecommand{\assertionname}{Assertion}

\providecommand{\remarkname}{Remark}
\providecommand{\notename}{Note}
\providecommand{\notationname}{Notation}
\providecommand{\claimname}{Claim}
\providecommand{\summaryname}{Summary}
\providecommand{\acknowledgmentname}{Acknowledgment}
\providecommand{\casename}{Case}
\providecommand{\conclusionname}{Conclusion}

\providecommand{\definitionname}{Definition}
\providecommand{\conditionname}{Condition}
\providecommand{\problemname}{Problem}
\providecommand{\examplename}{Example}
\providecommand{\exercisename}{Exercise}
\providecommand{\algorithmname}{Algorithm}
\providecommand{\questionname}{Question}
\providecommand{\axiomname}{Axiom}
\providecommand{\propertyname}{Property}
\providecommand{\assumptionname}{Assumption}
\providecommand{\hypothesisname}{Hypothesis}

\newcommand{\R}{{\mathbb R}}

\renewcommand{\d}{{\mathrm d}}

\makeatother

\begin{document}
\title{Identification of the residual term in multiplicative self-decomposition using Fox $H$-functions}
\author{\and \textbf{Mohamed Erraoui}\\
 Department of Mathematics, Faculty of Science, El Jadida,\\
 Choua{\"i}b Doukkali University, Morocco\\
 Email: erraoui.mohamed@ucd.ac.ma
\and \textbf{Jos\'e Lu{\'i}s da Silva},\\
 CIMA, Faculty of Exact Sciences and Engineering,\\
 University of Madeira, Campus da Penteada,\\
 9020-105 Funchal, Portugal.\\
 Email: joses@staff.uma.pt}
\date{\today \hspace{0.5cm}\currenttime}
\maketitle
\begin{abstract}
Multiplicative self-decomposable laws describe random variables that can be decomposed into a product of a scaled-down version of themselves and an independent residual term. Shanbhag et al.~(1977) have shown that the gamma distribution is multiplicative self-decom\-posable, in particular, the exponential distribution. As a result, they established the multiplicative self-decomposability of the absolute value of a centered normal random variable. A limitation of Shanbhag's result is that the distribution of the residual component is not explicitly identified. In this paper, we aim to fill this gap by providing an explicit distribution of the residual term using a Fox $H$-function. More precisely, the residual term follows an $M$-Wright distribution for the exponential distribution, whereas for the generalized gamma distribution and the absolute value of a centered normal random variable, an $H_{1,1}^{1,0}$ distribution with different parameters.  \\[.5cm]
AMS Classification 2020: Primary: 60E05, 60E07. Secondary: 44A20, 33C60, 33E12.
\end{abstract}
\tableofcontents{}

\section{Introduction}
A real random variable $X$ with distribution (or law) $\mu$ is said to be self-decomposable
(cf.~\cite[Definition~15.1]{Sato1999}) if, for any constant $\alpha\in(0,1)$,
there exists an independent random variable (called the residual term), denoted as $X_{\alpha}$, such
that
\begin{equation}\label{add-decomp}
X\overset{\mathcal{L}}{=}\alpha X+X_{\alpha}.
\end{equation}
The symbol $\overset{\mathcal{L}}{=}$ used here and below denotes equality in law. Specifically, the characteristic function $\hat{\mu}(z):=\mathbb{E}[\mathrm{e}^{\mathrm{i} z X}]$ satisfies
\[
  \hat{\mu}(z)=\hat{\mu}(\alpha z)\hat{\mu}_\alpha (z),
\] 
where $\mu_\alpha$ is the distribution of $X_\alpha$.  
Classic examples of self-decomposable distributions include Gaussian distributions on $\R^d$, gamma distributions, and Cauchy distributions on $\R$. We note that self-decomposable distributions are an important subclass of infinitely divisible laws and have unimodal associated densities; see \cite[p.~404]{Sato1999}.

In the paper \cite{Shanbhag.1977}, the authors established the self-decomposability of
the random variable $\log(Z)$ (resp.~$\log(|U|)$, where $Z$ is a gamma (resp.~$U$ is a centered normal) random variable.
This is equivalent to considering the ``multiplicative'' self-decomposability
of $Z$ and $|U|$, expressed as follows: for every $\alpha\in(0,1)$, we have 
\begin{equation}
Z\overset{\mathcal{L}}{=}Z^{\alpha} Z_\alpha ,\label{mult-decomp-Z}
\end{equation}
and
\begin{equation}
   |U|\overset{\mathcal{L}}{=} |U|^{\alpha}X_{\alpha}, \label{mult-decomp-U}
\end{equation}
 where $Z_\alpha$ and $X_{\alpha}$ are independent residual terms. In the particular case where $Z$ has an exponential distribution, the random variable $Z_\alpha$ takes the form
$Y_{\alpha}^{-\alpha}$, with $Y_{\alpha}$ having a stable one-sided
distribution with parameter $\alpha$; see Eq.~(2) in \cite{Shanbhag.1977}.

One key element missing from Shanbhag's work is the explicit form of the probability density function (pdf) of the residual term. This absence can be understood because, at the time, stable distributions generally lacked closed-form expressions. Only a few special cases, such as the normal, Cauchy, and Lévy distributions, were known to admit explicit formulas.

We take this opportunity to highlight an important yet often overlooked paper by Schneider \cite{Schneider1986}, in which the Mellin transform is used to represent stable distributions as Fox $H$-functions. Schneider demonstrated that stable densities arise as particular instances of Fox $H$-functions, thereby providing a unified representation of these distributions and offering a powerful analytical framework for effective use. Alternative formulations of stable densities can also be found in the works of Mainardi et al.~\cite{Mainardi:2003vn,Mainardi2006} and \cite{Kataria2018}. 
Our contribution in this paper aims to fill this gap and to show that the decompositions \eqref{mult-decomp-Z} and \eqref{mult-decomp-U} provide a characterization of $Z$ and $U$. To this end, we employ the Mellin transform. This tool is particularly well-suited to the study of products of independent random variables, in contrast to the characteristic function used in Shanbhag's work. This choice is motivated by the fact that the Mellin transform of the product of two independent random variables equals the product of their Mellin transforms. Consequently, the Mellin transform of the residual term is obtained as the quotient of the Mellin transform of the original distribution and the Mellin transform of its reduced form. To achieve the first aim, the primary task involves identifying the probability density function associated with this resulting transform. We find that the density of the residual term $Z_\alpha$ can be expressed in terms of the $M$-Wright function in the exponential case (cf.~Theorem~\ref{thm:self-decomp-exp}); whereas for the gamma distribution, it is expressed as a Fox $H$-function; see Theorem~\ref{thm:self-decomp-gamma} below. For the absolute value of a centered normal random variable, the density distribution of $Z_\alpha$ is expressed using the Wright function; see Theorem~\ref{thm:decomposition-Gaussian}. 

The second objective, which we consider particularly noteworthy, is that the decompositions given in \eqref{mult-decomp-Z} and \eqref{mult-decomp-U}, involving the aforementioned residual terms, lead to a characterization of the Gamma and standard Gaussian distributions. This result is obtained by solving an appropriate functional equation and by invoking the uniqueness property of the Mellin transform, see Propositions~\ref{prop:characterization-exp}, \ref{prop:characterization-gamma}, and \ref{prop:characterization-Gaussian} below.

The paper is organized as follows. In Section~\ref{sec:decomp-exp}, we recall the definition and some properties of the Mellin transform that will be necessary for subsequent discussions. In addition, the $M$-Wright density function and its key properties are reviewed. It will be the density of the residual term in the self-decomposition of the exponential random variable; cf.~Theorem~\ref{thm:self-decomp-exp}. We show that the self-decomposition \eqref{exp-decomp}, where the residual term follows an $M$-Wright density, characterizes the exponential distributions; see Proposition~\ref{prop:characterization-exp}. Section~\ref{sec:decomp-gamma} addresses the gamma case, where we show that the residual term has a density involving a Fox $H$-function; see Theorem~\ref{thm:self-decomp-gamma} and Proposition~\ref{prop:characterization-gamma}. In Section~\ref{sec:decomp-gaussian}, as a consequence of the result of Section~\ref{sec:decomp-gamma}, we express the density of the residual term in the self-decomposition of the absolute value of a centered normal random variable as a Wright function.

\section{Multiplicative self-decomposition of the exponential distribution\label{sec:decomp-exp}}
First, we review the tools required for this work, namely the Mellin transform and the $M$-Wright function, and recall their most relevant properties. For any sufficiently well-behaved function $f:\mathbb{R^{+}}\longrightarrow\R$ the Mellin transform of $f$ is defined by
\begin{equation}\label{eq:Mellin-transform}
    (\mathcal{M}f)(z):=\int_0^\infty f(x)x^{z-1}\,\mathrm{d}x,\quad z\in\mathbb{C},
\end{equation}
whenever the integral converges. 
The integral \eqref{eq:Mellin-transform} defines the Mellin transform in a vertical strip in the $z$-plane whose boundaries are determined by the analytic structure of $f(x)$ as $x \to 0^+$ and $x \to +\infty$.
For any function with the following asymptotic behavior   
\[
f(x) =
\begin{cases} 
O\left( x^{-a + \varepsilon} \right) & \text{as } x \to 0^+, \\
O\left( x^{-b - \varepsilon} \right) & \text{as } x \to +\infty,
\end{cases}
\]
for every (small) $\varepsilon > 0$ and $a < b$, the integral \eqref{eq:Mellin-transform} converges absolutely and defines an analytic function in the strip $a < \Re (z) < b$. This strip is known as \emph{strip of analyticity} of $(\mathcal{M}f)(\cdot)$, see Theorem 1.5.18 in  \cite{Glaeske2006}.
 
As discussed earlier, the Mellin transform is a fundamental tool in probability theory. Its primary application lies in characterizing the distributions of products of independent random variables via the Mellin convolution formula 
\begin{equation}\label{eq:Mellin-convolution}
(f\star g)(x):=\int_{0}^{\infty}f\left(\frac{x}{y}\right)g(y)\,\frac{\mathrm{d}y}{y}=\int_{0}^{\infty}f(y)g\left(\frac{x}{y}\right)\frac{\mathrm{d}y}{y},
\end{equation}
whenever the integral exists. To illustrate, given two independent random variables $X$ with density $\rho_X$ and $Y$ with density $\rho_Y$, the density $\rho_{XY}$ of the product $XY$ is given in terms of the Mellin convolution, that is, $\rho_{XY}=\rho_X\star\rho_Y$, see, for example, \cite{Zolotarev.1957, Springer.1979, Carter1977}. This is achieved through the following property: Let $f,g$ be given so that $f\star g$ is well-defined, then
\[
(\mathcal{M}(f\star g))(z)=(\mathcal{M}f)(z)(\mathcal{M}g)(z).
\]
Invoking the Mellin transform allows us to derive an equivalent representation of the self-decomposition as follows: let $X$ be a positive random variable satisfying Assumption (H) (below). We denote its Mellin transform by
\[
h(z)=\mathbb{E}[X^{z-1}],\quad  a<\Re(z)<b.
\]
So, $X$ is multiplicative self-decomposable if and only if for any $\alpha \in (0,1)$ the function
\[
F_\alpha (z)=\frac{h(z)}{h(\alpha(z-1)+1)},  a<\Re(z)<b,
\]
is the Mellin transform of a positive random variable $X_\alpha$. 

As we will see, the self-decomposability of the exponential distribution involves Wright functions, specifically the variant known as the Mainardi function.
The Wright function $W_{\lambda,\mu}(z)$, introduced and studied by
Wright in \cite{Wright1933,Wright1935, Wright1935a}, is defined by the series representation,  which converges for all complex $z$: 
\[
W_{\lambda,\mu}(z):=\sum_{n=0}^{\infty}\frac{z^{n}}{n!\,\Gamma(\lambda n+\mu)},\quad\lambda>-1,\quad\mu\in\mathbb{C}.
\]
This series defines an entire function on the complex plane.

In his study of the time-fractional diffusion-wave equation \cite{Mainardi.1993},
Main\-ardi introduced a special case of the Wright function, denoted
by $M_{\beta}$, defined as 
\[
M_{\beta}(z):=W_{-\beta,\,1-\beta}(-z)=\sum_{n=0}^{\infty}\frac{(-z)^{n}}{n!\,\Gamma(1-\beta(n+1))},\quad0<\beta<1.
\]
Here are some special and well-known examples:
\[
M_{1/2}(z)=\frac{1}{\sqrt{\pi}}e^{-\frac{z^{2}}{4}}\quad \text{and}\quad M_{1/3}(z)=3^{2/3}\,\mathrm{Ai}\left(3^{-1/3}z\right),
\]
where $\mathrm{Ai}(\cdot)$ is the Airy function. We point out that the function $M_{\beta}$ finds its most significant applications
when its argument is real. We present some important properties of this function.
\begin{itemize}
\item \textbf{Subordination formula:} For $x>0$, the following integral
relation holds: 
\[
M_{\alpha\beta}(x)=\int_{0}^{\infty}s^{\alpha}\,M_{\alpha}\left(\frac{x}{s^{\alpha}}\right)M_{\beta}(s)\,\d s,\quad0<\alpha,\beta<1.
\]
\item \textbf{Absolute moments:} The moments of $M_{\beta}$ are finite
and are given by 
\[
\int_{0}^{\infty}x^{\delta}\,M_{\beta}(x)\,dx=\frac{\Gamma(\delta+1)}{\Gamma(\beta\delta+1)},\quad\delta>-1,\quad0\leq\beta<1.
\]
\end{itemize}
Hence, for $\delta=0$ we obtain \begin{equation}\label{normalisation}
    \int_{0}^{\infty}M_{\beta}(x)\,dx=1.
\end{equation}
The Laplace transform of $M_{\beta}$ is expressed in terms of the
classical Mittag-Leffler function $E_{\beta}$: 
\begin{equation}\label{eq:LT-Mbeta}
\mathcal{L}\left(M_{\beta}\right)(s)=\int_{0}^{\infty}e^{-sx}M_{\beta}(x)\,dx=E_{\beta}(-s)=\sum_{n=0}^{\infty}\frac{(-s)^{n}}{\Gamma(\beta n+1)}.
\end{equation}
To interpret $M_{\beta}$ as a probability density function (pdf), the property of non-negativity is a necessary requirement. It is known (see \cite{Pollard48, Feller71}) that $E_{\beta}(-x)$, for $x>0$,
is \emph{completely monotone} if and only if $0<\beta\leq1$. Bernstein's
theorem implies that $M_{\beta}$ is a \emph{nonnegative function}.
Hence, with \eqref{normalisation}, we can consider it as a pdf supported on $[0,+\infty)$. 

A positive random variable $Y_\beta$ is said to follow an $M$-Wright distribution if its distribution is absolutely continuous with respect to the Lebesgue measure, with the density function $M_\beta(t),\,t\geq 0$.
For certain specific values of the parameter $\beta$, the associated random variable $Y_\beta$ can be identified. For example, when $\beta=1/2$, $Y_{1/2}$ coincides with the one-sided Gaussian random variable with density $M_{1/2}(t)=\pi^{-1/2}\mathrm{e}^{-t^2/4}$.

\begin{rem}\label{rem:limit-beta-zero}
    The exponential random variable, denoted by $Y_{0}$, can be regarded as a limiting case of the family $Y_\beta$ as $\beta\to 0^+$. This offers an alternative interpretation of Cressie's result (see \cite{cressie.1975} and also \cite{DuMouchel.1971}) and differs from the one presented in Shanbhag et al.~\cite{Shanbhag.1977}. It represents a weak continuity result.
\end{rem}

We provide the multiplicative self-decomposition of the exponential random
variable $Y_{0}$ in terms of the random variables $Y_{\beta}$ with $0<\beta < 1$. 

The following theorem yields the desired self-decomposition, which can also be regarded as an extension of Eq.~(6.4) in \cite{Mainardi:2003vn} to the case $\eta=0$. 
\begin{thm}\label{thm:self-decomp-exp}
For any $0<\beta <1$, the following self-decomposition holds 
\begin{equation}\label{exp-decomp}
Y_{0}\overset{\mathcal{L}}{=}Y_{0}^{\beta}Y_{\beta}. 
\end{equation}
\end{thm}

\begin{proof}
It is a well-known fact that $Y_0^\beta$ has a Weibull distribution with parameter $1/\beta$ and density $g_\beta (x)=\frac{1}{\beta}x^{1/\beta-1}\mathrm{e}^{-x^{1/\beta}}$, $x\geq 0$. The probability density of the product of two independent random variables is given by the Mellin convolution of the corresponding densities of the two variables. Hence, the probability density $\rho_{\beta}$ of $Y_{0}^{\beta}Y_{\beta}$ is expressed as
\begin{equation}
\rho_{\beta}(x)=(M_{\beta}\star g_\beta)(x),\quad x\geq 0.\label{eq:density_exp}
\end{equation}
It is not difficult to show that, for any $z\in\mathbb{C}$ such that $\Re(z)>0$, the Mellin transform of $g_\beta$ is
\[
(\mathcal{M}g_\beta)(z)=\int_{0}^{\infty}x^{z-1}g_\beta(x)\,\mathrm{d}x=\Gamma(\beta(z-1)+1),
\]
and the Mellin transform of $M_{\beta}$ is given by the generalized
moments of $M_{\beta}$, that is, 
\[
(\mathcal{M}M_{\beta})(z)=\int_{0}^{\infty}x^{z-1}M_{\beta}(x)\,\mathrm{d}x=\frac{\Gamma(z)}{\Gamma(\beta(z-1)+1)}.
\]
Thus, for any $z\in\mathbb{C}$ such that $\Re(z)>0$, we have
\[
(\mathcal{M}\rho_{\beta})(z)=(\mathcal{M}M_{\beta})(z)(\mathcal{M}g_\beta)(z)=\frac{\Gamma(z)\Gamma(\beta(z-1)+1)}{\Gamma(\beta(z-1)+1)}=\Gamma(z).
\]
On the other hand, the Mellin transform of the exponential distribution
is $\Gamma(z)$, $\Re(z)>0$. Therefore, by the uniqueness of the Mellin transform
(cf.~Theorem 1.5.22 in \cite{Glaeske2006} and Corollary~5 in \cite{Butzer.1997}), the claim follows. 
\end{proof}
\begin{rem}
An alternative method for proving this is to show that the Laplace transform
of the density $\rho_{\beta}$ matches the Laplace transform of the
exponential density with the help of \eqref{eq:LT-Mbeta} and the formula
\[
\int_{0}^{\infty}\mathrm{e}^{-x}E_{\beta}(-x^{\beta}t)\,\mathrm{d}x=\frac{1}{1+t},\quad\beta\geq0,\;t\ge0,
\]
see, for example, Eq.~(E.51) in \cite{Mainardi2010}.
\end{rem}
The self-decomposition \eqref{exp-decomp} characterizes the exponential distribution, as shown in the following proposition. Let $Y$ be a random variable with pdf 
$\rho_Y$ on $\mathbb{R}^+$ satisfying:

\noindent \textbf{Assumption (H)}.\label{assumptionH} Let $0<a<1<b$. The function $x\mapsto x^{z-1}\rho_{Y}(x),\;\; x\geq 0$, is integrable for every $z$ in the strip $a<\Re(z)<b$. 
\begin{prop}\label{prop:characterization-exp}
     If $Y$ satisfies \eqref{exp-decomp} for some $\beta \in (0,1)$, then 
     \[
     Y\overset{\mathcal{L}}{=}Y_{0}.
     \] 
\end{prop}
\begin{proof}
Since $Y$ fulfills assumption \textbf{(H)} and condition \eqref{exp-decomp}, the Mellin transform $(\mathcal{M}\rho_Y)$ can be expressed as 
\begin{equation}\label{eq:mellin}
    (\mathcal{M}\rho_Y)(z)=\frac{\Gamma(z)}{\Gamma(\beta(z-1)+1)}(\mathcal{M}\rho_Y)(\beta(z-1)+1),\quad 0<a<\Re(z)<b.
\end{equation}
Set $h(z):=\dfrac{(\mathcal{M}\rho_Y)(z)}{\Gamma(z)}$, $0<a<\Re(z)<b$. From Eq.~\eqref{eq:mellin}, we infer that $h$ satisfies the following functional equation
\begin{equation}\label{funct-eq}
h(z)=h(\beta (z-1) +1),\quad 0<a<\Re(z)<b.
\end{equation}
By iterating the functional equation, we conclude that for any integer $n\geq 2$, we have
$$
h(z)=h\left(\beta^n z +(1-\beta)\sum_{k=0}^{k=n-1} \beta^k\right),\quad 0<a<\Re(z)<b.
$$
Given the continuity of $h$, as $n$ approaches infinity, we obtain 
\[
h(z)=h(1)=1,
\]
for every $z$ such that $0<a<\Re(z)<b$. Theorem 1.5.22 in \cite{Glaeske2006}, on the uniqueness property of the Mellin transform, therefore, allows us to conclude that the random variables $Y$ and $Y_0$ have the same distribution.  
\end{proof}

\section{Multiplicative self-decomposition of the gamma distribution\label{sec:decomp-gamma}}
Let $r>0$ be given, and let $Z_r$ be a gamma random variable with density $\rho_{Z_r}(t) = \frac{1}{\Gamma(r)} t^{r-1} e^{-t}$, $t > 0$. 
The multiplicative self-decomposition of $Z_r$ uses the so-called Fox $H$-density function, which we introduce now.  According to standard notation (see \cite{Kilbas2004} for more details), the $H$-functions were introduced
by Fox in~\cite{Fox1961} and are defined by the Mellin--Barnes type integral, with the integrand containing products and quotients of the Euler gamma functions. Examples include Wright's generalized hypergeometric functions, $M$-Wright functions, Meijer $G$-functions, and the exponential function, among others. To establish our result, we shall need to consider the following particular Fox $H$-function $H^{1,0}_{1,1}$ defined via Mellin-Barnes integrals by
\begin{equation}\label{Barnes-resp}
H_{1,1}^{1,0}\left(z\,\middle|\,\genfrac{}{}{0pt}{}{(r-\alpha,\alpha)}{(r-1,1)}
\right):=\frac{1}{2\pi i}\int_{\mathscr{L}}\mathcal{H}_{1,1}^{1,0}(s)z^{-s}\,\mathrm{d}s, \quad 0<\alpha<1,
\end{equation}
where
\[
\mathcal{H}_{1,1}^{1,0}(s):=\frac{\Gamma(r-1+s)}{\Gamma(r-\alpha+\alpha s)},
\]
and $\mathscr{L}$ is a suitable contour in $\mathbb{C}$, $z^{-s}=\exp(-s\log|z|+\mathrm{i}\arg(z))$,
$\log|z|$ represents the natural logarithm of $|z|$ and $\arg(z)$
is not necessarily the principal value. The following proposition establishes the non-negativity of $\mathcal{H}_{1,1}^{1,0}$, which is a prerequisite for identifying it as a pdf.
\begin{prop}\label{eq: pdf-Fox}
    The Fox $H$-function $H_{1,1}^{1,0}\left(s\,\middle|\,\genfrac{}{}{0pt}{}{(r-\alpha,\alpha)}{(r-1,1)}\right)$ defines a probability density function on $\R^+$.    
\end{prop}
\begin{proof}
Since $a^{*}=\Delta=1-\alpha>0$, Theorems~1.2 and 2.2 in \cite{Kilbas2004} assert that the contour $\mathscr{L}$ in (\ref{Barnes-resp}) may be chosen as a left-oriented loop within a horizontal strip, beginning at the location $-\infty+\mathrm{i}$ and ending at $-\infty-\mathrm{i}$. Moreover, the Fox $H$-function defined in (\ref{Barnes-resp}) is an analytic
function on $\mathbb{C}\backslash\{0\}$ and is expressed as follows
\[
H_{1,1}^{1,0}(z)=\sum_{l=0}^{\infty}\text{Res}_{s=b_{l}}\left[\mathcal{H}_{1,1}^{1,0}(s)z^{-s}\right]=\sum_{l=0}^{\infty}\text{Res}_{s=b_{l}}\left[\frac{\Gamma(r-1+s)}{\Gamma(r-\alpha+\alpha s)}z^{-s}\right],
\]
where $b_{l}=-l+1-r$, $l\in\mathbb{N}$ are the poles of the function
$\Gamma(r-1+\cdot)$. It is easy to see that, for any $l\in\mathbb{N}$, we obtain
\[
\text{Res}_{s=b_{l}}\left[\frac{\Gamma(r-1+s)}{\Gamma(r-\alpha+\alpha s)}z^{-s}\right]=\frac{(-1)^{l}}{l!}\frac{z^{l+r-1}}{\Gamma\left(r(1-\alpha)-\alpha l\right)}.
\]
This leads to the following series expansion of the Fox $H$-function in \eqref{Barnes-resp}
\[
H_{1,1}^{1,0}\left(z\,\middle|\,\genfrac{}{}{0pt}{}{(r-\alpha,\alpha)}{(r-1,1)}\right)=z^{r-1}\sum_{l=0}^{\infty}\dfrac{(-1)^{l}}{l!\Gamma\left(r(1-\alpha)-\alpha l\right)}z^{l}.
\]
We note that the above series representation can be written in terms of the well-known Wright function $W_{\lambda,\mu}$ (see \cite{Wright1940b} or \cite{Mainardi2010}) as
\begin{equation}\label{eq:Relation-H-W}
H_{1,1}^{1,0}\left(z\,\middle|\,\genfrac{}{}{0pt}{}{(r-\alpha,\alpha)}{(r-1,1)}\right)=z^{r-1}W_{-\alpha,r(1-\alpha)}(-z).
\end{equation}
On the other hand, the Laplace transform of $W_{-\alpha,r(1-\alpha)}(\cdot)$ is given by the generalized Mittag-Leffler function (see \cite[Eq.~(F.25)]{Mainardi2010}), that is,
\begin{equation}\label{eq:LT-Wright-function}
\big(\mathcal{L}W_{-\alpha,r(1-\alpha)}(-\cdot)\big)(s)=E_{\alpha,r(1-\alpha)+\alpha}(-s),\; s>0.    
\end{equation}
For any $\alpha\in (0,1)$, it was shown by Schneider \cite{Schneider1996} that $E_{\alpha,r(1-\alpha)+\alpha}(-s)$ is completely monotone. Hence, the characterization of a completely monotonic function implies that $W_{-\alpha,r(1-\alpha)}(\cdot)$ is nonnegative on $\R^+$, which in turn ensures that $H_{1,1}^{1,0}\left(s\,\middle|\,\genfrac{}{}{0pt}{}{(r-\alpha,\alpha)}{(r-1,1)}\right)$ is also nonnegative on $\R^+$. 

The Mellin transform of $H_{1,1}^{1,0}\left(s\,\middle|\,\genfrac{}{}{0pt}{}{(r-\alpha,\alpha)}{(r-1,1)}\right)$ can be computed with the help of Theorem~2.2 in \cite{Kilbas2004}, valid on the open right half-plane $\Re(z) >1-r$, as
\begin{equation}\label{eq:Mellin-density}
\left(\mathcal{M}H_{1,1}^{1,0}\left(\cdot\,\middle|\,\genfrac{}{}{0pt}{}{(r-\alpha,\alpha)}{(r-1,1)}\right)\right)(z)=\mathcal{H}_{1,1}^{1,0}\left(z\,\middle|\,\genfrac{}{}{0pt}{}{(r-\alpha,\alpha)}{(r-1,1)}\right)=\frac{\Gamma(r-1+z)}{\Gamma(r-\alpha+\alpha z)}.
\end{equation}
Taking $z=1$ we can see that $H_{1,1}^{1,0}\left(s\,\middle|\,\genfrac{}{}{0pt}{}{(r-\alpha,\alpha)}{(r-1,1)}\right)$ is normalized. Thus, we may conclude that $H_{1,1}^{1,0}\left(s\,\middle|\,\genfrac{}{}{0pt}{}{(r-\alpha,\alpha)}{(r-1,1)}\right)$ is indeed a pdf on $\R^+$.
\end{proof}

For any $\alpha \in (0,1)$, let $Y_{\alpha,r}$ be a positive random variable (independent of $Z_r$) with density $\rho_{Y_{\alpha,r}}$ given by
\begin{equation}\label{eq:H-density}
\rho_{Y_{\alpha,r}}(t)=H_{1,1}^{1,0}\left(t\,\middle|\,\genfrac{}{}{0pt}{}{(r-\alpha,\alpha)}{(r-1,1)}\right),\quad t>0.    
\end{equation}

With this, we state the self-decomposition of the random variable $Z_r$.
\begin{thm}\label{thm:self-decomp-gamma}
 For any $\alpha\in (0,1)$, the following self-decomposition holds
\begin{equation}\label{eq:decomposition-gamma}
Z_{r}\overset{\mathcal{L}}{=}Z_{r}^{\alpha}Y_{\alpha,r}. 
\end{equation}
\end{thm}

\begin{proof}
    Let $\alpha\in (0,1)$ be given.  Thanks to the uniqueness property of the Mellin transform, we need to show that the Mellin transforms on each side of \eqref{eq:decomposition-gamma} coincide. In fact, the Mellin transforms of the densities $\rho_{Z_r}$ and $\rho_{Z_r^\alpha}$ are well-defined for any $z\in\mathbb{C}$ with $\Re(z)>1-r$ and are given by
    \[
    (\mathcal{M}\rho_{Z_r})(z)=\frac{\Gamma(r-1+z)}{\Gamma(r)}\qquad \mathrm{and}\qquad (\mathcal{M}\rho_{Z_r^\alpha})(z)=\frac{\Gamma(r-\alpha +\alpha z)}{\Gamma(r)}.    
    \]
Since the random variables $Z_r^\alpha$ and $Y_{\alpha,r}$ are independent, it follows that $$(\mathcal{M}\rho_{Z_r^\alpha Y_{\alpha,r}})(z)=(\mathcal{M}\rho_{Z_r^\alpha})(z)(\mathcal{M}\rho_{Y_{\alpha,r}})(z)$$ leading to
$$
(\mathcal{M}\rho_{Z_r}(z)=(\mathcal{M}\rho_{Z_r^\alpha Y_{\alpha,r}})(z),
$$
on the open right half-plane $\Re(z) > 1-r$.
This concludes the proof.
\end{proof}

We will show that Eq.~\eqref{eq:decomposition-gamma} characterizes the gamma distributions. 
 Let $Z$ be a random variable with a pdf $\rho_Z$ on $\mathbb{R}^+$ satisfying:

\noindent \textbf{Assumption (H')}. Let $c>1$ be given. The function $x\mapsto x^{z-1}\rho_{Z}(x)$, $x\geq 0$, is integrable for every $z$ in the strip $1-r<\Re(z)<c$. 

\begin{prop}\label{prop:characterization-gamma}
   If $Z$ satisfies \eqref{eq:decomposition-gamma} for a pair $(r,\alpha)\in (0,\infty)\times (0,1)$, then $$Z\overset{\mathcal{L}}{=}Z_r.$$
\end{prop}

\begin{proof}Under the Assumption~\textbf{(H')} and \eqref{eq:decomposition-gamma}, the Mellin transform $\mathcal{M}\rho_{Z}$ satisfies
\begin{equation}\label{eq:MT-Z}
(\mathcal{M}\rho_{Z})(z)=\frac{\Gamma(r-1+z)}{\Gamma(r-\alpha+\alpha z)}(\mathcal{M}\rho_{Z})(\alpha(z-1)+1),\quad 1-r<\Re(z)<c.
\end{equation}
Set $h(z):=(\mathcal{M}\rho_Z)(z)/\Gamma(z-1+r)$, $1-r<\Re(z)<c$. Then, Eq.~\eqref{eq:MT-Z} gives rise to the following functional equation
\[
h(z)=h(\alpha (z-1)+1), \quad 1-r<\Re(z)<c.
\]
As with the function in Eq.~\eqref{funct-eq}, we infer that
\[h(z)=h(1)=\frac{1}{\Gamma(r)}, \quad 1-r<\Re(z)<c.
\] 
Thus, we conclude that
\[
(\mathcal{M}\rho_Z)(z)=\frac{\Gamma(z+r-1)}{\Gamma(r)}=(\mathcal{M}\rho_{Z_r})(z),\quad 1-r<\Re(z)<c.
\]
The above, together with the uniqueness property of the Mellin transform, establishes the result of the proposition.
\end{proof}

\section{On the multiplicative self-decomposition of the Gaussian distribution\label{sec:decomp-gaussian}}
Shanbhag et al.~\cite{Shanbhag.1977} showed that the absolute value of a standard normal distribution possesses the property of self-decomposability. However, the residual random variable involved in the self-decomposition was not identified in their work. In this section, we identify the random variable as a consequence of the self-decomposition of the gamma distribution.

Let $U=\mathcal{N}(0,1)$ be a centered normal random variable. 
Since $U^{2}/2$ has a gamma distribution with parameter
$1/2$, then according to Theorem \ref{thm:self-decomp-gamma}, for any $\alpha\in(0,1)$, we have the self-decomposition \eqref{eq:decomposition-gamma}
for $U^{2}/2$, that is,
\[
U^{2}/2\overset{\mathcal{L}}{=}\left(U^{2}/2\right)^{\alpha}Y_{\alpha,1/2},
\]
with the density of is given by
\[
\rho_{Y_{\alpha,1/2}}(t)=H_{1,1}^{1,0}\left(t\,\middle|\,\genfrac{}{}{0pt}{}{(1/2-\alpha,\alpha)}{(-1/2,1)}\right),\quad t>0.
\]
Hence, it is easy to see that 
\[
\left|U\right|\overset{\mathcal{L}}{=}\left|U\right|^{\alpha}\,2^{(1-\alpha)/2}Y_{\alpha,1/2}^{1/2}.
\]
Now, after some calculations, we deduce that the residual random variable $X_{\alpha}:=2^{(1-\alpha)/2}Y_{\alpha,1/2}^{1/2}$ has the density function
\[
\rho_{X_{\alpha}}(t)=2^{\alpha}\,t\,H_{1,1}^{1,0}\left(2^{\alpha-1}\,t^{2}\,\middle|\,\genfrac{}{}{0pt}{}{(1/2-\alpha,\alpha)}{(-1/2,1)}\right),\quad t>0.
\]
It follows from Equation~\eqref{eq:Relation-H-W} that the density $\rho_{X_{\alpha}}$ can also be written in terms of the Wright function, namely
\begin{equation}\label{eq:density-Xalpha}
\rho_{X_{\alpha}}(t) =2^{(1+\alpha)/2}W_{-\alpha,(1-\alpha)/2}(-2^{\alpha-1}t^{2}),\quad t>0.
\end{equation}
We summarize the results above as follows:
\begin{thm}\label{thm:decomposition-Gaussian}
  Let $U$ be a centered normal random variable, and $\alpha\in(0,1)$ be given. The following self-decomposition holds
  \begin{equation}
|U|\overset{\mathcal{L}}{=}|U|^{\alpha}X_{\alpha},\label{eq:Guassian-decomposition1}
\end{equation}
where the random variable $X_\alpha$ is independent of $U$ and has a density function given by \eqref{eq:density-Xalpha}.
\end{thm}

Now we show that \eqref{eq:Guassian-decomposition1} characterizes the Gaussian distribution. To this end, we need the Mellin transforms of $\rho_{|\mathcal{N}(0,1)|}$ and $\rho_{X_\alpha}$. Using a change of variable formula, we obtain for any $z\in\mathbb{C}$ with $\Re(z)>0$ that 
\begin{equation} \label{eq:MT-Gaussian-distribution}
(\mathcal{M}\rho_{|\mathcal{N}(0,1)|})(z) =\sqrt{\frac{2}{\pi}}\int_0^\infty t^{z-1}\exp \left(-t^2/2\right)\mathrm{d}t =\frac{2^{z/2}}{\sqrt{2\pi}}\Gamma\left(\frac{z}{2}\right).
\end{equation}
From this, we deduce the Mellin transform of $|\mathcal{N}(0,1)|^{\alpha}$, that is, 
\[
(\mathcal{M}\rho_{|\mathcal{N}(0,1)|^\alpha})(z)=\frac{2^{(\alpha(z-1)+1)/2}}{\sqrt{2\pi}}\Gamma\left(\frac{\alpha(z-1)+1}{2}\right).
\]
Therefore, we obtain the Mellin transform of $\rho_{X_\alpha}$ via \eqref{eq:Guassian-decomposition1}, namely
\[
(\mathcal{M}\rho_{X_\alpha})(z)=2^{\frac{z-1}{2}(1-\alpha)}\frac{\Gamma\left(\frac{z}{2}\right)}{\Gamma\left(\frac{1}{2}+\alpha(\frac{z-1}{2})\right)}.
\]
\begin{prop}\label{prop:characterization-Gaussian}
Let $U$ be a symmetric random variable with
a continuous pdf $\rho_{U}$ on $\mathbb{R}$ satisfying Assumption~\textbf{(H)}. If $U$
verifies (\ref{eq:Guassian-decomposition1}) for some $\alpha\in(0,1)$, then
\[
U\overset{\mathcal{L}}{=}\mathcal{N}(0,1).
\]
\end{prop}
\begin{proof}
Using the independence of $|U|$ and $X_\alpha$ in \eqref{eq:Guassian-decomposition1}, the Mellin transform of the density $\rho_{|U|}$ satisfies
\begin{equation*}
(\mathcal{M}\rho_{|U|})(z)=2^{\frac{z-1}{2}(1-\alpha)}\frac{\Gamma\left(\frac{z}{2}\right)}{\Gamma\left(\frac{1}{2}+\alpha(\frac{z-1}{2})\right)}(\mathcal{M}\rho_{|U|})(\alpha(z-1)+1),
\end{equation*}
for every $z\in\mathbb{C}$ such that $0<a<\Re(z)<b$. 
Within the same framework as in the preceding cases, and employing the auxiliary function $h$, which is defined on the strip $0 < a < \Re(z) < b$, we have
\[
h(z):=\frac{(\mathcal{M}\rho_{|U|})(z)}{\Gamma\left(\frac{z}{2}\right)2^{\frac{z-1}{2}}},
\]
yields
\[
(\mathcal{M}\rho_{|U|})(z)=\frac{1}{\sqrt{2\pi}}\Gamma\left(\frac{z}{2}\right)2^{\frac{z}{2}}.
\]
Consequently, from \eqref{eq:MT-Gaussian-distribution} and the fact that $U$ is symmetric, we obtain
\[
U\overset{\mathcal{L}}{=}\mathcal{N}(0,1). \qedhere
\]
\end{proof}

\section{Generalized gamma distribution}
Formally introduced by E.~W.~Stacy in 1962~\cite{Stacy1962}, the generalized gamma (GG) distribution family is defined on $[0,\infty)$ by its pdf as 
\begin{equation}\label{GG}
f(t)= \mathrm{C} t^{a}\exp(-kt^{b}),\quad k>0,\ b>0,\ a>-1,  
\end{equation}
where $\mathrm{C}=\frac{bk^{(a+1)/b}}{\Gamma((a+1)/b)}$ is the normalizing constant. This distribution family provides a powerful, unifying framework for a variety of commonly used lifetime distributions. In particular, it includes several well-known special cases, such as the gamma, exponential, Weibull, and Rayleigh distributions. Furthermore, with appropriate parameter choices, it captures the Chi ($\chi$) and Chi-square ($\chi^{2}$) distributions, as well as specific functions of standard normal variables, such as their positive even powers, their modulus, and all positive powers of their modulus. The GG distribution can effectively represent data exhibiting various hazard rate patterns, including increasing, decreasing, bathtub-shaped, and unimodal behaviors. This property is useful in reliability and lifetime analysis.  We aim to prove that this model represents a broad class of multiplicative self-decomposable distributions.

The Mellin transform of $f$ is 
\[
h(z)=(\mathcal{M}f)(z)=\mathrm{C} \int_{0}^{\infty}t^{z+a-1}e^{-kt^{b}}\,\d t,
\]
provided that $\Re(s)>-a$. Using the change of variable
$u=kt^{b}$, we obtain

\[
h(z)=\mathrm{C} \int_{0}^{\infty}\left(\frac{u}{k}\right)^{\frac{s+a-1}{b}}e^{-u}\cdot\frac{1}{b}k^{-1/b}u^{\frac{1}{b}-1}\,\d u.
\]
Thus, for every $\operatorname{Re}(z)>-a$, the Mellin transform $h$ is given by
\[
h(z)=\mathrm{C}  \frac{k^{-(z+a)/b}}{b}\int_{0}^{\infty}u^{\frac{z+a}{b}-1}e^{-u}du=\mathrm{C}  \frac{k^{-(z+a)/b}}{b}\Gamma\left(\frac{z+a}{b}\right).
\]
To verify whether $f$ is multiplicatively self-decomposable, we must determine whether, for any $\alpha \in (0,1)$, the function 
\[
	F_\alpha(z)  =\frac{h(z)}{h(\alpha(z-1)+1)}
	 =k^{(\alpha-1)(z-1)/b}\frac{\Gamma\left(\frac{z+a}{b}\right)}{\Gamma\left(\frac{\alpha z+a+1-\alpha}{b}\right)}
\]
can serve as the Mellin transform of a probability density function. It is well known from the
properties of the Fox $H$-function that for $\operatorname{Re}(z)>-a$
we have 
\begin{equation*}
	F_\alpha(z) =\mathcal{M}\left(k^{(1-\alpha)/b} H_{1,1}^{1,0}\left(k^{(1-\alpha)/b}(\cdot)\,\middle|\,\genfrac{}{}{0pt}{}{((a+1-\alpha)/b,\alpha/b)}{(a/b,1/b)}\right)\right)(z).
	\end{equation*}
Hence, it remains to verify the non negativity of $H_{1,1}^{1,0}\left(t\,\middle|\,\genfrac{}{}{0pt}{}{((a+1-\alpha)/b,\alpha/b)}{(a/b,1/b)}\right)$ for $t\geq 0$.
Following the steps used in the proof of Proposition~\ref{eq: pdf-Fox}, we can express it in terms of the $M$-Wright function as follows 
\begin{align*}
	H_{1,1}^{1,0}\left(z\,\middle|\,\genfrac{}{}{0pt}{}{((a+1-\alpha)/b,\alpha/b)}{(a/b,1/b)}\right)\
	&=z^{a}\sum_{l=0}^{\infty}\frac{(-z^{b})^{l}}{l!\Gamma(-l\alpha+(a+1)(1-\alpha)/b)}\\
	&=z^{a}W_{-\alpha,(a+1)(1-\alpha)/b}(-z^{b}).
\end{align*}
To achieve our goal, we must ensure that the function $W_{-\alpha,(a+1)(1-\alpha)/b}(-t)$, $t\geq 0$, be nonnegative. This condition is related, via Schneider's result, to the complete monotonicity of its Laplace transform, which is expressed by the two-parameter Mittag-Leffler function, as follows:
\begin{equation}
	\bigl(\mathcal{L}W_{-\alpha,(a+1)(1-\alpha)/b}(-\cdot)\bigr)(s)=E_{\alpha,\alpha+(a+1)(1-\alpha)/b}(-s),\;s>0.
\end{equation}
According to Schneider, the function $E_{\alpha,\alpha+(a+1)(1-\alpha)/b}(-s)$
is completely mo\-notone for any $\alpha\in(0,1)$ and $a\geq-1$. Hence, $W_{-\alpha,(a+1)(1-\alpha)/b}(-t)$, $t\geq 0$, is nonnegative, leading to the conclusion that 
\[k^{(1-\alpha)/b} H_{1,1}^{1,0}\left(k^{(1-\alpha)/b}t\,\middle|\,\genfrac{}{}{0pt}{}{((a+1-\alpha)/b,\alpha/b)}{(a/b,1/b)}\right),\quad t\geq 0,
\]
is a pdf. Hence, we have the following result.
\begin{thm}\label{thm:decomposition-GG}
  Let $X$ be a random variable with pdf \eqref{GG}, and $\alpha\in(0,1)$ be given. The following self-decomposition holds
  \begin{equation}\label{eq:GG-decomposition}
X\overset{\mathcal{L}}{=}X^{\alpha}X_{\alpha},
\end{equation}
where the random variable $X_\alpha$ is independent of $X$ and has a density function $\rho_X$ given by 
\begin{equation}\label{eq:self-decom-GG}
\rho_X(t)=k^{(1-\alpha)/b} H_{1,1}^{1,0}\left(k^{(1-\alpha)/b}t\,\middle|\,\genfrac{}{}{0pt}{}{((a+1-\alpha)/b,\alpha/b)}{(a/b,1/b)}\right),\;t\ge0.
\end{equation}
\end{thm}

By proceeding analogously to the previous cases, we may characterize GG random variables that exhibit the multiplicative self-decomposition property given in \eqref{eq:GG-decomposition}.  
\begin{prop}\label{prop:characterization-GG}
Let $X$ be a nonnegative random variable with a continuous pdf $\rho_{X}$ satisfying Assumption~\textbf{(H)}. If $X$
verifies (\ref{eq:GG-decomposition}) for some $\alpha\in(0,1)$, then
\[
X\overset{\mathcal{L}}{=}\mathrm{GG}.
\]
\end{prop}

\section{Discussion and conclusion\label{sec:discussion}}
We have identified the density distributions of the residual terms in the self-decomposition considered by Shanbhag et al.~\cite{Shanbhag.1977},  thereby strengthening both their original results and their extension to the GG class. These densities are expressed in terms of special functions, specifically the product of a monomial and Wright's functions. In the special case of the exponential distribution, the density is represented by an $M$-Wright function. For the generalized gamma distributions and the distribution of the absolute value of a normal random variable, the probability density can be expressed in terms of a Fox $H$-function, which may itself be represented as a product of a monomial and a Wright function. In addition, we show that all self-decompositions characterize the corresponding distributions.

The general problem of multiplicative self-decomposability can be stated as follows. Let $X$ be a positive random variable satisfying Assumption (H); see page~\pageref{assumptionH}. We denote its Mellin transform by
\[
h(z)=\mathbb{E}[X^{z-1}],\quad  a<\Re(z)<b.
\]
So, $X$ is multiplicative self-decomposable if and only if the function
\[
F(z)=\frac{h(z)}{h(\alpha(z-1)+1)}
\]
is the Mellin transform of a positive random variable $Y$.
For instance, assume that $F$ meets the assumptions of Theorem~4.3.6 in \cite{Zamanian1968}; namely, $F$ is analytic on the strip $\{z \mid a' < \Re(z) < b'\}$ and satisfies
\[
\lvert F(z)\rvert \le K \lvert z\rvert^{-2}
\]
for some positive constant $K$. Then, independent of the choice $a < \gamma < b$, the function $f$ defined by
\[
f(x):=\frac{1}{2\pi\mathrm{i}}\int_{\gamma-\mathrm{i}\infty}^{{\gamma+\mathrm{i}\infty}} F(z)x^{-z}\,\d z,
\]
is continuous on $(0,\infty)$ and satisfies
\[
(\mathcal{M}f)(z)=F(z),\; \text{for at least } a'<\Re(z)<b'.
\]
The challenging task is to identify the conditions under which the function $f$ is nonnegative.

\subsection*{Funding}

The Center for Research in Mathematics and Applications (CIMA) has
partially funded this work related to Statistics, Stochastic Processes
and Applications (SSPA) group, through the grant UID/4674/2025 
of FCT-Funda{\c c\~a}o para a Ci{\^e}ncia e a Tecnologia, Portugal.

\bibliographystyle{alpha}

\bibliography{luis}

\end{document}